\documentclass[11pt,reqno,a4paper]{amsart}

\usepackage[T1]{fontenc}
\usepackage[utf8]{inputenc}

\usepackage{libertinus}
\usepackage{microtype}

\usepackage{geometry}

\usepackage{xcolor}

\usepackage{amsmath,amssymb,amsthm,amsfonts}
\usepackage{mathtools}
\usepackage{mathrsfs}
\usepackage{backref}

\usepackage{graphicx}

\usepackage[
colorlinks=true,
linkcolor=blue,
citecolor=blue,
urlcolor=blue
]{hyperref}

\usepackage[nameinlink,noabbrev]{cleveref}

\usepackage{titlesec}

\titleformat{\section}
  {\normalfont\large\bfseries\centering}
  {\thesection.}
  {0.75em}
  {}

\titlespacing*{\section}
  {0pt}
  {3.2ex plus 1ex minus .2ex}
  {1.6ex plus .3ex}

\titleformat{\subsection}
  {\normalfont\normalsize\bfseries}
  {\thesubsection.}
  {0.75em}
  {}

\titlespacing*{\subsection}
  {0pt}
  {2.4ex plus .8ex minus .2ex}
  {1.0ex plus .2ex}

\titleformat{\subsubsection}
  {\normalfont\normalsize\bfseries\itshape}
  {\thesubsubsection.}
  {0.75em}
  {}

\titlespacing*{\subsubsection}
  {0pt}
  {2.0ex plus .6ex minus .2ex}
  {0.7ex plus .2ex}

\numberwithin{equation}{section}

\usepackage{aliascnt}

\theoremstyle{plain}
\newtheorem{theorem}{Theorem}[section]

\newaliascnt{proposition}{theorem}
\newtheorem{proposition}[proposition]{Proposition}
\aliascntresetthe{proposition}

\newaliascnt{lemma}{theorem}
\newtheorem{lemma}[lemma]{Lemma}
\aliascntresetthe{lemma}

\newaliascnt{corollary}{theorem}

\aliascntresetthe{corollary}

\newtheorem{maintheorem}{Theorem}

\newtheorem{mainobstruction}{Obstruction}

\theoremstyle{definition}
\newaliascnt{definition}{theorem}

\aliascntresetthe{definition}

\theoremstyle{remark}
\newaliascnt{remark}{theorem}
\newtheorem{remark}[remark]{Remark}
\aliascntresetthe{remark}

\usepackage{etoolbox}

\AtBeginEnvironment{proof}{\vspace{6pt}}

\usepackage{enumitem}

\setlist{
itemsep=2pt,
topsep=4pt
}

\usepackage[toc,page]{appendix}

\newcommand{\SL}{\operatorname{SL}}

\newcommand{\cH}{\mathcal{H}}

\newcommand{\cP}{\mathcal{P}}

\newcommand{\cT}{\mathcal{T}}

\newcommand{\bC}{\mathbb{C}}

\newcommand{\bN}{\mathbb{N}}

\newcommand{\bQ}{\mathbb{Q}}
\newcommand{\bR}{\mathbb{R}}

\newcommand{\bZ}{\mathbb{Z}}

\newcommand\subsetsim{\mathrel{%
\ooalign{\raise0.2ex\hbox{$\subset$}\cr\hidewidth\raise-0.8ex\hbox{\scalebox{0.9}{$\sim$}}\hidewidth\cr}}}

\DeclareMathOperator{\Hom}{Hom}

\DeclareMathOperator{\supp}{supp}
\DeclareMathOperator{\trace}{tr}

\DeclareMathOperator{\Mat}{Mat}

\DeclareMathOperator{\Vol}{Vol}

\DeclareMathOperator{\Rat}{Rat}

\usepackage{comment}

\makeatletter
\renewcommand{\l@section}[2]{%
  \par\addpenalty\@secpenalty
  \addvspace{1.0em plus 1pt}%
  \@tempdima 1.5em
  \begingroup
    \parindent \z@
    \rightskip \@pnumwidth
    \parfillskip -\@pnumwidth
    \leavevmode
    \bfseries
    \advance\leftskip\@tempdima
    \hskip -\leftskip
    #1\nobreak\hfil \nobreak\hb@xt@\@pnumwidth{\hss #2}\par
  \endgroup}
\makeatother

\usepackage{bbm}
\usepackage{amscd}
\usepackage[all,cmtip]{xy}
\usepackage{changepage}
\usepackage{calc}
\usepackage{scalerel}
\usepackage{extarrows}
\usepackage{aliascnt}

\DeclareMathSymbol{\shortminus}{\mathbin}{AMSa}{"39}

\newcommand{\ignore}[1]{}

\DeclareMathOperator{\Eig}{Eig}

\DeclareMathOperator{\Span}{Span}

\DeclareMathOperator{\VolSpec}{VolSpec}
\DeclareMathOperator{\TraceSpec}{TraceSpec}
\DeclareMathOperator{\EhrSpec}{EhrSpec}

\begin{document}

\title[Trace spectra of simplices in large sets]{Trace spectra of simplices in large sets}

\author[M. Bj\"orklund]{Michael Bj\"orklund}
\address{Department of Mathematics, Chalmers University of Technology and University of Gothenburg, Gothenburg, Sweden}
\email{micbjo@chalmers.se}

\author[A. Fish]{Alexander Fish}
\address{School of Mathematics and Statistics, University of Sydney, NSW 2006, Australia}
\email{alexander.fish@sydney.edu.au}

\author[S. Sanadhya]{Shrey Sanadhya}
\address{School of Mathematics and Statistics, University of Sydney, NSW 2006, Australia}
\email{shrey.sanadhya@sydney.edu.au}


\begin{abstract}
Given an ordered tuple \(\mathbf v=(v_0,\ldots,v_d)\) of vectors in \(\bR^d\), let
\(A_{\mathbf v}=[\,v_1-v_0\ \cdots\ v_d-v_0\,]\) be its edge matrix. We prove
that, in every finite colouring of \(\bR^d\), one colour class realizes every
prescribed value of the higher characteristic coefficients
\[
        (c_2(A_{\mathbf v}),\ldots,c_d(A_{\mathbf v})).
\]
This extends Graham's theorem on volumes, which corresponds to the last
coefficient \(c_d(A_{\mathbf v})=\det(A_{\mathbf v})\). We also prove a
discrete analogue: if \(E\subseteq\bZ^d\) has positive upper Banach density,
then, for some \(q\geq 1\), the set of coefficient tuples realized by ordered
tuples in \(E\) contains
\[
        q^2\bZ\times q^3\bZ\times\cdots\times q^d\bZ .
\]
Finally, we show that the ordinary trace \(c_1(A_{\mathbf v})\) cannot be added
to these conclusions. The proof combines a quantitative directional expansion
result for ergodic actions of free abelian groups with a trace calculation for
a family of model edge matrices.
\end{abstract}

\maketitle

\section{Introduction} 
\subsection{Volume spectra and Graham's theorem}

Let \(E\subseteq\bR^d\). The \emph{volume spectrum} of \(E\) is the set of volumes of
\(d\)-simplices with vertices in \(E\):
\[
        \VolSpec(E)
        =
        \{\Vol(\Delta_{\mathbf v}):
        \mathbf v=(v_0,\ldots,v_d)\in E^{d+1}
        \text{ is affinely independent}\},
\]
where $\Delta_{\mathbf v} = \operatorname{conv}\{v_0,\ldots,v_d\}$ denotes the convex hull. If $v_1-v_0,\ldots,v_d-v_0$ are linearly independent, we say that $\Delta_{\mathbf v}$ is a \emph{$d$-simplex}.

Equivalently, if
\[
        A_{\mathbf v}
        =
        [\,v_1-v_0\ \cdots\ v_d-v_0\,],
\]
is the \emph{edge matrix} of $\mathbf v$, then
\[
        \Vol(\Delta_{\mathbf v})
        =
        \frac{1}{d!}|\det(A_{\mathbf v})|.
\]
Graham proved that, in every finite colouring of \(\bR^d\), one colour class
contains \(d\)-simplices of every prescribed volume \cite{Graham_1980}.

\begin{theorem}[Graham]\label{Thm: Graham_volume}
Let \(d\geq 1\), and let
\[
        \bR^d=C_1\sqcup\cdots\sqcup C_r
\]
be a finite colouring. Then there exists \(1\leq i_0\leq r\) such that
\[
        \VolSpec(C_{i_0})=(0,\infty).
\]
\end{theorem}

The point is that the colour class \(C_{i_0}\) is chosen once and then realizes
all prescribed volumes. Related formulations and variants of Graham's theorem
have since been studied; see, for instance, \cite{Kovac_2023}.

A discrete analogue concerns subsets of \(\bZ^d\) of positive upper Banach
density. For \(E\subseteq\bZ^d\), the upper Banach density of \(E\) is
\[
        d^*(E)
        =
        \limsup_{N\to\infty}
        \sup_{t\in\bZ^d}
        \frac{|E\cap(t+[-N,N]^d)|}{|[-N,N]^d|}.
\]
For every \(\delta>0\) and \(d\geq 1\), there exists
\(q=q(\delta,d)\in\bN\) such that whenever \(d^*(E)\geq\delta\),
\[
        q\bZ\subseteq d! \cdot\VolSpec(E).
\]
Thus the volume spectrum of a positive-density subset of \(\bZ^d\) contains an
infinite arithmetic progression, after the standard normalization by \(d!\).
The qualitative statement was proved in \cite{Bjorklund_Fish_2024}; the
uniform dependence of \(q\) on \(\delta\) and \(d\) follows from
\cite{Fish_Skinner_2025}.

Since volume is the determinant of the edge matrix, Graham's theorem can be
viewed as a statement about the last coefficient of the characteristic
polynomial of \(A_{\mathbf v}\). The purpose of this paper is to study what
happens when one asks for more of this characteristic data.

\subsection{Trace spectra and main results}

Let \(A\in\Mat_d(\bR)\). We write its characteristic polynomial as
\[
        \det(xI-A)
        =
        x^d-c_1(A)x^{d-1}+c_2(A)x^{d-2}
        -\cdots+(-1)^d c_d(A).
\]
Recall that
\[
        c_k(A)=\trace(\wedge^k A),
        \qquad 1\leq k\leq d,
\]
where \(\wedge^k A\) denotes the induced map on \(\wedge^k\bR^d\). Equivalently,
\(c_k(A)\) is the sum of the \(k\times k\) principal minors of \(A\). In
particular,
\[
        c_1(A)=\trace(A),
        \qquad
        c_d(A)=\det(A).
\]

Graham's theorem concerns the last coefficient
\(c_d(A_{\mathbf v})=\det(A_{\mathbf v})\). In this paper we consider the
coefficients
\[
        c_2(A_{\mathbf v}),\ldots,c_d(A_{\mathbf v}).
\]
We write
\[
        \mathbf{Tr}(A)=(c_2(A),\ldots,c_d(A)).
\]
The coefficient \(c_1(A)\), the ordinary trace, is not included. This omission
is necessary for the main theorems to hold; see \Cref{obs:euclidean-trace} and \Cref{obs:discrete-trace} below.

Let \(E\subseteq\bR^d\). For any ordered tuple
\(\mathbf v=(v_0,\ldots,v_d)\in E^{d+1}\), not necessarily affinely
independent, we call
\[
        A_{\mathbf v}
        :=
        [\,v_1-v_0\ \cdots\ v_d-v_0\,]\in\Mat_d(\bR)
\]
the \emph{edge matrix} of \(\mathbf v\). The \emph{trace spectrum} of \(E\) is
\[
        \TraceSpec(E)
        =
        \{\mathbf{Tr}(A_{\mathbf v}) :
        \mathbf v=(v_0,\ldots,v_d)\in E^{d+1}\}.
\]
Thus the trace spectrum is defined for arbitrary ordered tuples, not only for
ordered \(d\)-simplices. The ordering is part of the definition:
\(\mathbf{Tr}(A_{\mathbf v})\) is not an invariant of the underlying unordered
set of vertices.

Our first main result is the following extension of Graham's theorem.

\begin{maintheorem}\label{thm:euclidean-main}
Let \(d\geq 2\), and let
\[
        \bR^d=C_1\sqcup\cdots\sqcup C_r
\]
be a finite colouring. Then there exists \(1\leq i_0\leq r\) such that
\[
        \TraceSpec(C_{i_0})=\bR^{d-1}.
\]
\end{maintheorem}

In other words, one colour class realizes every prescribed tuple
\[
        (a_2,\ldots,a_d)\in\bR^{d-1}
\]
as
\[
        (c_2(A_{\mathbf v}),\ldots,c_d(A_{\mathbf v}))
\]
for some ordered tuple \(\mathbf v\in C_{i_0}^{d+1}\). If the prescribed tuple satisfies \(a_d\neq 0\), then the corresponding edge
matrix has non-zero determinant. Hence the ordered tuple
\(\mathbf v=(v_0,\ldots,v_d)\) is affinely independent and spans a
\(d\)-simplex.

\begin{remark}
No measurability assumption is made on the colour classes in
\Cref{thm:euclidean-main}. This is the same level of generality as Graham's
finite-colouring theorem. It is stronger in this respect than many analytic
Euclidean Ramsey and density results, where measurable colour classes or
measurable positive-density sets are part of the hypotheses; see, for instance,
\cite{Kovac_2023}.
\end{remark}

For positive-density subsets of \(\bZ^d\), we prove the following discrete
version.

\begin{maintheorem}\label{thm:discrete-main}
Let \(d\geq 2\), and let \(E\subseteq\bZ^d\) satisfy \(d^*(E)>0\). Then there
exists \(q\in\bN\) such that
\[
        q^2\bZ\times q^3\bZ\times\cdots\times q^d\bZ
        \subseteq
        \TraceSpec(E).
\]
Moreover, \(q\) may be chosen depending only on the rank \(d\) and on the upper Banach density \(d^*(E)\).
\end{maintheorem}

The exponents \(2,\ldots,d\) come from homogeneity:
\[
        c_k(qA)=q^k c_k(A).
\]
Thus the subgroup in \Cref{thm:discrete-main} is the trace spectrum one expects
from simplices whose edge vectors lie in \(q\bZ^d\).

\subsection{Relation to Ehrhart spectra}

There is another natural way to refine the volume spectrum in the lattice
setting. If \(\Delta\subseteq\bR^d\) is a lattice simplex, its Ehrhart
polynomial is defined by
\[
        P_\Delta(t)=|t\Delta\cap\bZ^d|,
        \qquad t\in\bN.
\]
The leading coefficient of \(P_\Delta\) is the Euclidean volume of \(\Delta\).
Thus the Ehrhart polynomial contains the volume, but also further arithmetic
information about the lattice simplex. For \(E\subseteq\bZ^d\), define its
\emph{Ehrhart spectrum} by
\[
        \EhrSpec(E)
        =
        \{P_{\Delta_{\mathbf v}} :
        \mathbf v=(v_0,\ldots,v_d)\in E^{d+1}
        \text{ and } v_0,\ldots,v_d
        \text{ are affinely independent}\}.
\]

In \cite{Bjorklund_Cullman_Fish_2025}, the Ehrhart spectrum of
positive-density subsets of \(\bZ^d\) was studied. The main result there says
that if \(E\subseteq\bZ^d\) has positive upper Banach density, then there
exists \(q\in\bN\) such that
\[
        \EhrSpec(q\bZ^d)\subseteq\EhrSpec(E).
\]

The present paper concerns a different refinement of volume. Ehrhart
polynomials are invariant under translations by vectors in \(\bZ^d\) and under
the natural action of \(\SL_d(\bZ)\). This invariance is central in
\cite{Bjorklund_Cullman_Fish_2025}: it allows one to move lattice simplices by
\(\SL_d(\bZ)\) without changing their Ehrhart polynomials.

The trace spectrum behaves differently. Let
\(\mathbf v=(v_0,\ldots,v_d)\in(\bZ^d)^{d+1}\). If
\(\gamma\in\SL_d(\bZ)\), then
\[
        A_{\gamma\mathbf v}=\gamma A_{\mathbf v}.
\]
The coefficients \(c_k\) are invariant under conjugation, but not under left
multiplication. Hence the tuple
\(\mathbf{Tr}(A_{\mathbf v})\) is not preserved by the natural
\(\SL_d(\bZ)\)-action. This is the basic difference from the Ehrhart setting:
there is no corresponding symmetry which organizes the trace spectrum.

\subsection{The linear trace obstruction}

The ordinary trace \(c_1(A)=\trace(A)\) is not included in the definition of
\(\mathbf{Tr}(A)\). This is not just a feature of the proof. The analogues of
Theorems A and B fail if \(c_1\) is included.

\begin{mainobstruction}\label{obs:euclidean-trace}
There is a measurable colouring of \(\bR^2\) with \(49\) colours such that, for every
colour class \(C\), the set
\[
        \{\trace(A_{\mathbf v}) :
        \mathbf v=(v_0,v_1,v_2)\in C^3
        \text{ and } v_0,v_1,v_2
        \text{ are affinely independent}\}
\]
is not all of  \(\bR\). Consequently, for every \(d\geq 2\), there is a finite
measurable colouring of \(\bR^d\) for which no colour class realizes all values of
\(c_1(A_{\mathbf v})\).
\end{mainobstruction}

\begin{mainobstruction}\label{obs:discrete-trace}
For every \(d\geq 2\), there is a set \(E\subseteq \bZ^d\) with positive upper
Banach density such that
\[
        \{\trace(A_{\mathbf v}) :
        \mathbf v=(v_0,\ldots,v_d)\in E^{d+1}
        \text{ and } v_0,\ldots,v_d
        \text{ are affinely independent}\}
\]
does not contain any finite-index subgroup of \(\bZ\).
\end{mainobstruction}

\subsection{Ideas of the proof}

We briefly describe the proof of \Cref{thm:euclidean-main} and
\Cref{thm:discrete-main}. The main step is a dynamical trace-spectrum theorem
for actions of groups of the form
\[
        \Gamma^{d-1}\times\bZ,
\]
where \(\Gamma<\bR\) is finitely generated. The discrete theorem follows from
this dynamical result by the correspondence principle. The Euclidean colouring
result is then obtained by applying the discrete theorem inside suitable
finitely generated subgroups of \(\bR^d\), and using Graham's product theorem
to make the colour independent of the prescribed trace tuple.

The algebraic input is a family of model edge matrices. We introduce matrices
\(M_{m,n,t}\) and \(A_s\), and compute the coefficients
\[
        c_2(A_sM_{m,n,t}),\ldots,c_d(A_sM_{m,n,t}).
\]
The parameters \(n=(n_1,\ldots,n_{d-1})\) enter these coefficients in a simple
way. More importantly, the coefficients \(c_2,\ldots,c_d\) are independent of
the parameter \(t\). Thus \(t\) can later be chosen from recurrence, while the
trace tuple is controlled by \(n\). Once the other parameters are fixed, varying
\(n\) gives a finite-index subgroup of the target trace lattice. This is the
mechanism which produces the prescribed trace values.

The ergodic input is a directional expansion theorem. The model matrices above
require recurrence along directions of the form
\[
        (s_1,\ldots,s_{d-1},1).
\]
We prove that, after passing to a bounded finite-index subgroup and to an
ergodic component, some direction of this form has almost full orbit saturation
of the set under consideration. This is proved by a dichotomy: either such
directional expansion already holds, or the spectral measure has enough
rational mass to pass to a smaller ergodic component on which the measure of
the set increases. Since this increase cannot happen indefinitely, the
iteration stops after a bounded number of steps.

The directional expansion theorem is then combined with a quantitative
Poincar\'e recurrence argument to supply
the first \(d-1\) columns of the model
matrix, while the directional expansion supplies the last column.
As a result, for
every parameter \(n\) in a suitable finite-index subgroup, one obtains
\[
        \mu\bigl(
        B\cap Qe_1.B\cap\cdots\cap Qe_d.B
        \bigr)>0
\]
for a matrix \(Q=Q(s,m,n,t)\). The trace calculation for the model matrices
then gives the dynamical trace-spectrum theorem.

Finally, the correspondence principle turns the dynamical theorem into
\Cref{thm:discrete-main}. To prove \Cref{thm:euclidean-main}, one fixes a
target tuple \(a=(a_2,\ldots,a_d)\in\bR^{d-1}\), chooses a finitely generated
subgroup \(\Gamma<\bR\) containing suitable rescalings of its coordinates, and
applies the discrete theorem inside \(\Gamma^{d-1}\times\bZ\). 
This shows that, for each target tuple, some colour class realizes it. Graham's
product theorem is then used to choose one colour class which works for all
target tuples simultaneously.

\subsection{Acknowledgments}
M.B. was supported by the Swedish Research Council VR 11253322. A.F. and S.S.
would like to thank the Australian Research Council for support through the
grant DP240100472.

\setcounter{tocdepth}{2}
\tableofcontents

\section{Preliminaries}


\subsection{Free abelian groups and their duals}

All abelian groups in the paper are written additively. Let \(\Lambda\) be a
free abelian group of rank \(R\). After choosing a \(\bZ\)-basis
\(\mathscr B=(\beta_1,\ldots,\beta_R)\), we identify \(\Lambda\) with
\(\bZ^R\) by
\[
        (m_1,\ldots,m_R)\mapsto \sum_{j=1}^R m_j\beta_j .
\]
A vector \(\lambda\in\Lambda\) is called \emph{primitive} if it can be
completed to a \(\bZ\)-basis of \(\Lambda\). Equivalently, if
\(\lambda=\smash{\sum_{j=1}^R m_j\beta_j}\), then \(\lambda\) is primitive if and only
if \(\gcd(m_1,\ldots,m_R)=1\). We denote the set of primitive vectors in
\(\Lambda\) by \(\cP_\Lambda\).

Let \(S^1\subset\bC^\times\) be the unit circle and let
\[
        \widehat{\Lambda}=\Hom(\Lambda,S^1)
\]
be the Pontryagin dual of \(\Lambda\). We define the \emph{rational spectrum}
of \(\Lambda\) by
\[
        \Rat(\Lambda)
        =
        \{\chi\in\widehat{\Lambda}:
        \chi|_{\Lambda_0}=1
        \text{ for some finite-index subgroup }\Lambda_0<\Lambda\}.
\]
Equivalently, \(\Rat(\Lambda)\) is the set of finite-order characters in
\(\widehat{\Lambda}\). For \(M\in\bN\), we also write
\[
        \Rat_M(\Lambda)
        =
        \{\chi\in\widehat{\Lambda}:
        \chi^n=1 \text{ for some } 1<n\leq M\}.
\]
Thus \(\Rat_M(\Lambda)\subseteq \Rat(\Lambda)\setminus\{1\}\).

If \(H<\Lambda\) is a subgroup, its \emph{annihilator} is
\[
        H^\perp
        =
        \{\chi\in\widehat{\Lambda}:\chi(\lambda)=1
        \text{ for all }\lambda\in H\}.
\]
For \(\lambda\in\Lambda\), we write \(L_\lambda=\bZ\lambda\). Then
\[
        L_\lambda^\perp
        =
        \{\chi\in\widehat{\Lambda}:\chi(\lambda)=1\}.
\]
Finally, we recall the notion of a haystack. An infinite subset
\(\cH\subseteq \cP_\Lambda\) is called a \emph{haystack} if every \(R\)-tuple
of distinct elements of \(\cH\) generates a finite-index subgroup of
\(\Lambda\). Since \(\Lambda\) has rank \(R\), this is equivalent to requiring
that any \(R\) distinct elements of \(\cH\) are linearly independent over
\(\bQ\).


\subsection{Spectral measures, rational spectrum, and measure increments}

Let \(\Lambda\curvearrowright (X,\mu)\) be an ergodic measure-preserving action
on a standard Borel probability space, written as
\[
        (\lambda,x)\mapsto \lambda.x .
\]
For a measurable set \(B\subset X\), we write
\[
        \lambda.B=\{\lambda.x:x\in B\}.
\]

Let \(B\subset X\) be measurable with \(\mu(B)>0\). By Bochner's theorem, there
is a unique finite positive Borel measure \(\sigma_B\) on
\(\widehat{\Lambda}\) such that
\[
        \mu(B\cap \lambda.B)
        =
        \int_{\widehat{\Lambda}}\chi(\lambda)\,d\sigma_B(\chi),
        \qquad \lambda\in\Lambda .
\]
We call \(\sigma_B\) the \emph{spectral measure} of \(B\). If we want to
emphasize the dependence on \(\mu\), we write \(\sigma_{\mu,B}\).

The following lower bound for cyclic expansions is
\cite[Lemma 2.5]{Bjorklund_Fish_2024}.

\begin{lemma}\label{Lem: sigma_B}
Let \(\Lambda\curvearrowright (X,\mu)\) be an ergodic measure-preserving action
and let \(B\subset X\) be measurable with \(\mu(B)>0\). Then for every
\(\lambda\in\Lambda\),
\[
        \mu\left(\bigcup_{n\in\bZ} n\lambda.B\right)
        >
        \frac{\mu(B)^2}{\sigma_B(L_\lambda^\perp)} .
\]
\end{lemma}

We shall use the following form of the ergodic decomposition for finite-index
subactions.

\begin{proposition}[\(k\Lambda\)-ergodic components]
\label{Prop: erg_com_k}
{\cite[Proposition A.2]{Bulinski_2017}}
Let \(\Lambda\curvearrowright (X,\mu)\) be an ergodic measure-preserving action.
Then for every \(k\in\bN\), there exist finitely many \(k\Lambda\)-invariant
and \(k\Lambda\)-ergodic probability measures \(\nu_1,\ldots,\nu_N\), with
disjoint supports, such that
\[
        \mu=\frac1N\sum_{j=1}^N\nu_j .
\]
Moreover, each \(\nu_j\) is of the form
\[
        \nu_j(\,\cdot\,)
        =
        \frac{\mu(\,\cdot\,\cap C_j)}{\mu(C_j)}
\]
for some \(k\Lambda\)-invariant set \(C_j\subset X\). We call
\(\nu_1,\ldots,\nu_N\) the \emph{\(k\Lambda\)-ergodic components} of \(\mu\).
\end{proposition}

For \(\chi\in\widehat{\Lambda}\), let \(\Eig_\Lambda(\chi)\) denote the space
of all \(f\in L^2(X,\mu)\) such that
\[
        f(\lambda.x)=\chi(\lambda)f(x),
        \qquad \lambda\in\Lambda,
\]
for \(\mu\)-almost every \(x\in X\). If \(R\subseteq\widehat{\Lambda}\), define
\[
        \Eig_\Lambda(R)
        =
        \overline{\Span\{f\in L^2(X,\mu):
        f\in\Eig_\Lambda(\chi)
        \text{ for some }\chi\in R\}}.
\]
We use the inner product
\[
        \langle f,g\rangle=\int_X f\overline g\,d\mu .
\]
If \(\chi\neq\xi\), then
\(\Eig_\Lambda(\chi)\perp\Eig_\Lambda(\xi)\), and ergodicity implies that each
\(\Eig_\Lambda(\chi)\) has dimension at most one. The following lemma is
standard, see e.g. Lemma 5.2 in \cite{Fish_Skinner_2025}.

\begin{lemma}\label{Lem: Proj_Eig}
Let \(\Lambda\curvearrowright (X,\mu)\) be an ergodic measure-preserving action
and let \(B\subset X\) be measurable with \(\mu(B)>0\). Let \(\sigma_B\) be the
spectral measure of \(B\). For \(\chi\in\widehat{\Lambda}\), let
\(P_{\Eig_\Lambda(\chi)}\) denote the orthogonal projection onto
\(\Eig_\Lambda(\chi)\). Then
\[
        \bigl\langle P_{\Eig_\Lambda(\chi)}1_B,1_B\bigr\rangle
        =
        \sigma_B(\{\chi\})
\]
for every \(\chi\in\widehat{\Lambda}\). In particular,
\[
        \sigma_B(\{1\})=\mu(B)^2,
\]
where \(1\) denotes the trivial character.
\end{lemma}

The next lemma is the measure-increment argument from
\cite[Lemma 3.3]{Fish_Skinner_2025}, see also \cite{Bulinski_Fish}. The idea goes back to
Roth \cite{Roth_1953}.

\begin{lemma}[Measure increment argument]\label{Lem: measure_increment}
Let \(\Lambda\curvearrowright (X,\mu)\) be an ergodic measure-preserving action
and let \(B\subset X\) be measurable with \(\mu(B)>0\). Let \(\sigma_B\) be the
spectral measure of \(B\). Then for every \(M\in\bN\), there exist
\(k=k(M)\in\bN\) and a \(k\Lambda\)-ergodic component \(\nu\) of \(\mu\) such
that
\[
        \nu(B)
        \geq
        \sqrt{\mu(B)^2+\sigma_B(\Rat_M(\Lambda))}.
\]
\end{lemma}

\begin{remark}
The integer \(k\) in \Cref{Lem: measure_increment} depends only on \(M\), and
not on the set \(B\). This uniformity will be used below.
\end{remark}

We shall also use the following quantitative form of Poincaré recurrence; see
\cite[p.7--8]{Fish_Skinner_2025}.

\begin{lemma}[Quantitative Poincaré recurrence]\label{Lem: Poincare_Quant}
Let \(\delta>0\). Then there exists
\[
        \ell_0=\ell_0(\delta)
        \leq
        \left\lfloor\frac{2}{\delta}\right\rfloor
\]
such that for every measure-preserving action $\bZ\curvearrowright (X,\mu)$
and every measurable set \(B\subset X\) with \(\mu(B)>\delta\), there exists
a non-zero integer \(\ell\), with \(|\ell|\leq \ell_0\), such that
\[
        \mu(B\cap \ell.B)>\frac{\delta^2}{2}.
\]
\end{lemma}


\subsection{Upper Banach density and the correspondence principle}

We shall use the following form of Furstenberg's correspondence principle \cite{Furstenberg_1977,Furstenberg_1981} for
actions of countable abelian groups.
If \(\Lambda\) is a countable abelian
group, a sequence \((F_N)_{N\in\bN}\) of finite subsets of \(\Lambda\) is called
a \emph{F{\o}lner sequence} if, for every \(\lambda\in\Lambda\),
\[
        \lim_{N\to\infty}
        \frac{|(F_N+\lambda)\triangle F_N|}{|F_N|}
        =
        0.
\]
For \(E\subseteq\Lambda\), define the \emph{upper Banach density} $d_\Lambda^*(E)$ by
\[
        d^*_\Lambda(E)
        =
        \sup_{(F_N)}
        \limsup_{N\to\infty}
        \frac{|E\cap F_N|}{|F_N|},
\]
where the supremum is taken over all F{\o}lner sequences in \(\Lambda\).

\begin{proposition}[Furstenberg correspondence principle]
\label{Prop: FCP_Lambda}
Let \(\Lambda\) be a countable abelian group, and let
\(E\subseteq\Lambda\) satisfy \(d^*_\Lambda(E)>0\). Then there exist an ergodic
measure-preserving action of \(\Lambda\) on a standard Borel probability space
\((X,\mu)\) and a measurable set \(B\subset X\), with
\(\mu(B)=d^*_\Lambda(E)\), such that for every \(r\in\bN\) and every
\(\lambda_1,\ldots,\lambda_r\in\Lambda\),
\[
        \mu\bigl(
        B\cap \lambda_1.B\cap\cdots\cap \lambda_r.B
        \bigr)
        \leq
        d^*_\Lambda\bigl(
        E\cap(E-\lambda_1)\cap\cdots\cap(E-\lambda_r)
        \bigr).
\]
\end{proposition}


\section{The algebraic trace model}
\label{sec:algebraic-trace-model}

\subsection{The model matrices}

Let \(d\geq 2\), and let \(\Gamma<\bR\) be a finitely generated subgroup. In
this section we introduce a family of matrices whose characteristic
coefficients will later be used to produce prescribed trace tuples.

Let
\[
        m=(m_1,\ldots,m_{d-1})\in\bZ^{d-1},\qquad
        n=(n_1,\ldots,n_{d-1})\in\Gamma^{d-1},
        \qquad
        t\in\bZ .
\]
Define
\[
        M_{m,n,t}
        =
        \begin{pmatrix}
        0 & 0 & 0 & \cdots & 0 & n_1 \\
        m_1 & 0 & 0 & \cdots & 0 & n_2 \\
        0 & m_2 & 0 & \cdots & 0 & n_3 \\
        \vdots & \vdots & \ddots & \ddots & \vdots & \vdots \\
        0 & 0 & \cdots & m_{d-2} & 0 & n_{d-1} \\
        0 & 0 & \cdots & 0 & m_{d-1} & t
        \end{pmatrix}.
\]
Thus \(M_{m,n,t}\) has \(m_1,\ldots,m_{d-1}\) on the subdiagonal and
\((n_1,\ldots,n_{d-1},t)\) as its last column.

For \(s=(s_1,\ldots,s_{d-1})\in\Gamma^{d-1}\), define
\[
        A_s
        =
        \begin{pmatrix}
        1 & 0 & \cdots & 0 & s_1 \\
        0 & 1 & \cdots & 0 & s_2 \\
        \vdots & & \ddots & & \vdots \\
        0 & 0 & \cdots & 1 & s_{d-1} \\
        0 & 0 & \cdots & 0 & 1
        \end{pmatrix}.
\]
We shall study the matrices
\[
        Q(s,m,n,t):=A_sM_{m,n,t}.
\]
A direct multiplication gives
\[
        Q(s,m,n,t)
        =
        \begin{pmatrix}
        0      & 0      & \cdots & 0       & s_1m_{d-1}
               & n_1+s_1t \\
        m_1    & 0      & \cdots & 0       & s_2m_{d-1}
               & n_2+s_2t \\
        0      & m_2    & \cdots & 0       & s_3m_{d-1}
               & n_3+s_3t \\
        \vdots & \vdots & \ddots & \vdots  & \vdots
               & \vdots \\
        0      & 0      & \cdots & m_{d-2} & s_{d-1}m_{d-1}
               & n_{d-1}+s_{d-1}t \\
        0      & 0      & \cdots & 0       & m_{d-1}
               & t
        \end{pmatrix}.
\]
The columns of \(Q(s,m,n,t)\) will later be used as edge vectors. The role of
the parameter \(n\) is to vary the trace tuple, while \(m\), \(s\), and \(t\)
control the lattice on which these tuples lie. The calculation below makes this
explicit.


\subsection{Higher-order traces of the model matrices}

We now compute the characteristic coefficients of \(Q(s,m,n,t)\). For
\(A\in\Mat_d(\bR)\), write
\[
        \det(xI-A)
        =
        x^d-c_1(A)x^{d-1}+c_2(A)x^{d-2}
        -\cdots+(-1)^d c_d(A).
\]
Equivalently,
\[
        c_k(A)
        =
        \sum_{\substack{I\subseteq\{1,\ldots,d\}\\ |I|=k}}
        \det(A_I),
        \qquad 1\leq k\leq d,
\]
where \(A_I\) denotes the principal submatrix of \(A\) with rows and columns in
\(I\).

The following calculation is the only one about the matrices \(Q(s,m,n,t)\)
that we shall need. The proof is given in \Cref{sec: trace_calc}.

\begin{lemma}\label{Lem: ck-formula-Q}
Let \(d\geq 2\), let
\[
        m=(m_1,\ldots,m_{d-1})\in\bZ^{d-1},\qquad
        n=(n_1,\ldots,n_{d-1})\in\Gamma^{d-1},
\]
let \(t\in\bZ\), and let
\(s=(s_1,\ldots,s_{d-1})\in\Gamma^{d-1}\). Let
\[
        Q=Q(s,m,n,t).
\]
Then, for every \(2\leq k\leq d-1\),
\[
        c_k(Q)
        =
        (-1)^{k-1}
        \bigl(n_{d-k+1}+s_{d-k}m_{d-k}\bigr)
        \prod_{j=d-k+1}^{d-1}m_j,
\]
and
\[
        c_d(Q)
        =
        \det(Q)
        =
        (-1)^{d-1}n_1\prod_{j=1}^{d-1}m_j .
\]
In particular, \(c_2(Q),\ldots,c_d(Q)\) are independent of \(t\).
\end{lemma}


\subsection{A finite-index trace lattice}

We now record the consequence of \Cref{Lem: ck-formula-Q} which will be used
later. Let \(\Gamma<\bR\) be a finitely generated subgroup. For
\(m=(m_1,\ldots,m_{d-1})\in\bZ^{d-1}\), define
\[
        \Lambda_m
        =
        m_{d-1}\Gamma
        \times
        m_{d-2}m_{d-1}\Gamma
        \times\cdots\times
        \left(\prod_{j=1}^{d-1}m_j\right)\Gamma
        \subseteq \Gamma^{d-1}.
\]

\begin{lemma}\label{Lem: finite_index_trace_lattice}
Let \(d\geq 2\), let \(\Gamma<\bR\) be a finitely generated subgroup, and let
\[
m=(m_1,\ldots,m_{d-1})\in\bZ^{d-1}
\]
satisfy \(m_j\neq 0\) for every \(j\).
Then for every \(s\in\Gamma^{d-1}\) and every \(t\in\bZ\),
\[
        \left\{
        (c_2(Q),\ldots,c_d(Q)):
        Q=Q(s,m,n,t),\ n\in\Gamma^{d-1}
        \right\}
        =
        \Lambda_m .
\]
In particular, this set is independent of \(s\) and \(t\).
\end{lemma}

\begin{proof}
By \Cref{Lem: ck-formula-Q}, for \(2\leq k\leq d-1\),
\[
        c_k(Q(s,m,n,t))
        =
        (-1)^{k-1}
        \bigl(n_{d-k+1}+s_{d-k}m_{d-k}\bigr)
        \prod_{j=d-k+1}^{d-1}m_j,
\]
and
\[
        c_d(Q(s,m,n,t))
        =
        (-1)^{d-1}n_1\prod_{j=1}^{d-1}m_j .
\]
As \(n\) ranges over \(\Gamma^{d-1}\), each shifted variable
\[
        n_{d-k+1}+s_{d-k}m_{d-k},
        \qquad 2\leq k\leq d-1,
\]
also ranges over all of \(\Gamma\). Thus the possible values of \(c_k\) are
exactly
\[
        \left(\prod_{j=d-k+1}^{d-1}m_j\right)\Gamma,
        \qquad 2\leq k\leq d,
\]
up to signs. This gives the claimed product set \(\Lambda_m\). The formula is
independent of \(t\), and the image set is independent of \(s\), as claimed.
\end{proof}


\section{Directional expansion}

The main result of this section is the following directional expansion theorem.

\begin{theorem}[Quantitative directional expansion]\label{Thm: Quant_direc_exp}
Let \(\delta>0\) and \(\varepsilon>0\). Let \(d\geq 2\), and let
\(\Lambda=\Gamma^{d-1}\times\bZ\), where \(\Gamma<\bR\) is a finitely generated
subgroup. Then there exists
\[
        k_0=k_0(\delta,\varepsilon,\operatorname{rank}(\Gamma),d)\in\bN
\]
such that the following holds: For every ergodic measure-preserving action
\(\Lambda\curvearrowright (X,\mu)\) and every measurable set \(B\subset X\)
with \(\mu(B)\geq\delta\), there exist an integer \(1\leq k\leq k_0\), a
\(k\Lambda\)-ergodic component \(\nu\) of \(\mu\), and elements
\(s_1,\ldots,s_{d-1}\in\Gamma\) such that \(\nu(B)\geq\mu(B)\) and
\[
        \nu\left(
        \bigcup_{\tau\in k\bZ} \tau(s_1,\ldots,s_{d-1},1).B
        \right)
        >
        1-\varepsilon .
\]
\end{theorem}

\subsection{The directional dichotomy}

The proof of \Cref{Thm: Quant_direc_exp} is based on the following dichotomy.

\begin{proposition}\label{Prop: Dicotomy}
Let \(\delta>0\) and \(\varepsilon>0\). Let \(d\geq 2\), and let
\(\Lambda=\Gamma^{d-1}\times\bZ\), where \(\Gamma<\bR\) is a finitely generated
subgroup. Then there exists
\[
        M=M(\delta,\varepsilon,\operatorname{rank}(\Gamma),d)\in\bN
\]
such that the following holds. Let
\(\Lambda\curvearrowright (X,\mu)\) be an ergodic measure-preserving action,
and let \(B\subset X\) be measurable with \(\mu(B)\geq\delta\). If
\[
        \sigma_B(\Rat_M(\Lambda))
        <
        \frac{\delta^2\varepsilon^2}{4},
\]
then there exist \(s_1,\ldots,s_{d-1}\in\Gamma\) such that
\[
        \mu\left(
        \bigcup_{\tau\in\bZ} \tau(s_1,\ldots,s_{d-1},1).B
        \right)
        >
        1-\varepsilon .
\]
\end{proposition}

We now define a haystack in \(\Lambda\) consisting of elements whose last
coordinate is \(1\). Let \(r=\operatorname{rank}(\Gamma)\), and let
\(\mathscr A=\{a_1,\ldots,a_r\}\) be a \(\bZ\)-basis of \(\Gamma\). Then
\(R:=\operatorname{rank}(\Lambda)=r(d-1)+1\). We use the associated basis
\[
        \mathscr B
        =
        \{a_je_i:1\leq i\leq d-1,\ 1\leq j\leq r\}
        \cup\{e_d\}
\]
of \(\Lambda\), where \(e_1,\ldots,e_d\) are the standard coordinate vectors.
If \(\lambda=\sum_{j=1}^R n_j(\lambda)\beta_j\) with respect to this basis, define
\[
        \|\lambda\|_{\mathscr B,\infty}
        =
        \max_{1\leq j\leq R}|n_j(\lambda)|.
\]

\begin{proposition}\label{Prop: Haystack_s_i}
Let \(\Gamma<\bR\) be a finitely generated subgroup of rank \(r\), let
\(d\geq 2\), and set \(\Lambda=\Gamma^{d-1}\times\bZ\). Choose a basis
\(\{a_1,\ldots,a_r\}\) of \(\Gamma\), and define
\[
\mathscr H
=
\left\{
h_t
:=
\left(
\sum_{j=1}^r t^j a_j,\;
\sum_{j=1}^r t^{r+j} a_j,\;
\ldots,\;
\sum_{j=1}^r t^{(d-2)r+j}a_j,\;
1
\right)
:\ t\in\bZ
\right\}.
\]
Then \(\mathscr H\) is a haystack in \(\Lambda\).
\end{proposition}

\begin{proof}
With respect to the basis \(\mathscr B\), the element \(h_t\) has coordinate
vector
\[
        (t,t^2,\ldots,t^{R-1},1).
\]
Hence \(h_t\) is primitive. It remains to show that every \(R\) distinct
elements of \(\mathscr H\) generate a finite-index subgroup of \(\Lambda\). Let
\(h_{t_1},\ldots,h_{t_R}\) be distinct elements of \(\mathscr H\). Then
\(t_1,\ldots,t_R\) are distinct integers. Let \(C\) be the \(R\times R\)
integer matrix whose columns are the coordinate vectors of these elements in
the basis \(\mathscr B\):
\[
C=
\begin{pmatrix}
t_1 & t_2 & \cdots & t_R\\
t_1^2 & t_2^2 & \cdots & t_R^2\\
\vdots & \vdots & \ddots & \vdots\\
t_1^{R-1} & t_2^{R-1} & \cdots & t_R^{R-1}\\
1 & 1 & \cdots & 1
\end{pmatrix}.
\]
Moving the last row to the first row gives a Vandermonde matrix. Hence
\[
        \det(C)
        =
        \pm\prod_{1\leq i<j\leq R}(t_j-t_i),
\]
which is non-zero. Therefore the subgroup generated by
\(h_{t_1},\ldots,h_{t_R}\) has full rank in \(\Lambda\), and hence finite
index. This proves that \(\mathscr H\) is a haystack.
\end{proof}

For \(M\geq 1\), define
\[
        \cH_M
        =
        \left\{
        \lambda\in\mathscr H:
        \|\lambda\|_{\mathscr B,\infty}
        \leq
        \left(\frac{M}{R!}\right)^{1/R}
        \right\}.
\]

\begin{lemma}\label{lem: H_M_independence}
The cardinality of \(\cH_M\) depends only on \(M\) and \(R\). In particular, it
is independent of the choice of basis of \(\Gamma\). Moreover, for fixed
\(R\), one has
\[
        |\cH_M|\to\infty
        \qquad\text{as } M\to\infty .
\]
\end{lemma}

\begin{proof}
With respect to the basis \(\mathscr B\), the coordinate vector of \(h_t\) is
\[
        (t,t^2,\ldots,t^{R-1},1).
\]
Hence
\[
        \|h_t\|_{\mathscr B,\infty}
        =
        \max\{|t|,|t|^2,\ldots,|t|^{R-1},1\}.
\]
Thus \(h_t\in\cH_M\) if and only if
\[
        |t|
        \leq
        \left(\frac{M}{R!}\right)^{1/(R(R-1))}.
\]
Consequently,
\[
        |\cH_M|
        =
        2\left\lfloor
        \left(\frac{M}{R!}\right)^{1/(R(R-1))}
        \right\rfloor+1 .
\]
This depends only on \(M\) and \(R\), and tends to infinity with \(M\) for fixed
\(R\).
\end{proof}

\begin{lemma}\label{Lem: annilator}
Let \(d\geq 2\), let \(\Lambda=\Gamma^{d-1}\times\bZ\), and let
\(R=\operatorname{rank}(\Lambda)\). If
\(\lambda_1,\ldots,\lambda_R\in\cH_M\) are distinct, then
\[
        \bigcap_{i=1}^R L_{\lambda_i}^{\perp}
        \subseteq
        \Rat_M(\Lambda)\cup\{1\}.
\]
\end{lemma}

\begin{proof}
Let \(\lambda_1,\ldots,\lambda_R\in\cH_M\) be distinct. Since \(\mathscr H\) is
a haystack, these elements generate a finite-index subgroup of \(\Lambda\).
Write
\[
        \lambda_i=\sum_{j=1}^R d_{ij}\beta_j,
        \qquad d_{ij}\in\bZ,
\]
and let \(D=(d_{ij})\in\Mat_R(\bZ)\). Since the \(\lambda_i\) are independent,
\(\det(D)\neq 0\). Since each \(\lambda_i\in\cH_M\),
\[
        |d_{ij}|
        \leq
        \left(\frac{M}{R!}\right)^{1/R}.
\]
Hence
\[
        1\leq |\det(D)|
        \leq
        R!\left(\left(\frac{M}{R!}\right)^{1/R}\right)^R
        =
        M .
\]
Let \(\chi\in\bigcap_{i=1}^R L_{\lambda_i}^{\perp}\), and set
\(n=|\det(D)|\). The subgroup \(L=\langle \lambda_1,\ldots,\lambda_R\rangle\)
has index \(n\). Hence \(n\lambda\in L\) for every \(\lambda\in\Lambda\). Thus
there are integers \(m_1,\ldots,m_R\) such that
\[
        n\lambda=\sum_{i=1}^R m_i\lambda_i .
\]
Applying \(\chi\), we get
\[
        \chi(\lambda)^n
        =
        \prod_{i=1}^R \chi(\lambda_i)^{m_i}
        =
        1 .
\]
Since this holds for every \(\lambda\in\Lambda\), either \(\chi=1\), or
\(\chi\) has order \(1<n\leq M\). We conclude that
\(\chi\in\Rat_M(\Lambda)\cup\{1\}\).
\end{proof}

\begin{lemma}\label{Lem: bound_L_gamma}
Let \(d\geq 2\), let \(\Lambda=\Gamma^{d-1}\times\bZ\), and let
\(R=\operatorname{rank}(\Lambda)\). For every \(\eta>0\), there exists
\(M=M(\eta,R)\in\bN\) such that the following holds. If \(\sigma\) is a finite
Borel measure on \(\widehat{\Lambda}\) with
\[
        \sigma(\widehat{\Lambda})\leq 1
        \qquad\text{and}\qquad
        \sigma(\Rat_M(\Lambda))\leq \frac{\eta}{2},
\]
then there exists \(\lambda\in\cH_M\) such that
\[
        \sigma(L_\lambda^\perp\setminus\{1\})
        \leq
        \eta .
\]
\end{lemma}

\begin{proof}
For \(\lambda\in\cH_M\), set
\[
        E_\lambda
        =
        L_\lambda^\perp\setminus(\Rat_M(\Lambda)\cup\{1\}).
\]
We claim that
\[
        \sum_{\lambda\in\cH_M}1_{E_\lambda}\leq R-1 .
\]
Indeed, if a character \(\chi\) belonged to
\(E_{\lambda_1},\ldots,E_{\lambda_R}\) for distinct
\(\lambda_1,\ldots,\lambda_R\in\cH_M\), then
\[
        \chi\in\bigcap_{i=1}^R L_{\lambda_i}^\perp .
\]
By \Cref{Lem: annilator}, this implies
\[
        \chi\in\Rat_M(\Lambda)\cup\{1\},
\]
contrary to the definition of the sets \(E_{\lambda_i}\).

Therefore,
\[
        |\cH_M|\min_{\lambda\in\cH_M}\sigma(E_\lambda)
        \leq
        \sum_{\lambda\in\cH_M}\sigma(E_\lambda)
        =
        \int_{\widehat{\Lambda}}
        \sum_{\lambda\in\cH_M}1_{E_\lambda}\,d\sigma
        \leq
        R-1 .
\]
By \Cref{lem: H_M_independence}, we may choose \(M=M(\eta,R)\) so large that
\[
        \frac{R-1}{|\cH_M|}
        \leq
        \frac{\eta}{2}.
\]
For this \(M\), there is \(\lambda\in\cH_M\) such that
\[
        \sigma(E_\lambda)\leq\frac{\eta}{2}.
\]
Since
\[
        L_\lambda^\perp\setminus\{1\}
        \subseteq
        E_\lambda\cup\Rat_M(\Lambda),
\]
we get
\[
        \sigma(L_\lambda^\perp\setminus\{1\})
        \leq
        \sigma(E_\lambda)+\sigma(\Rat_M(\Lambda))
        \leq
        \eta .
\]
\end{proof}

\begin{proof}[Proof of \Cref{Prop: Dicotomy}]
Let \(B\subset X\) satisfy \(\mu(B)\geq\delta\), and set
\[
        \eta=\frac{\delta^2\varepsilon^2}{2}.
\]
Then
\[
        \frac{\mu(B)^2}{\mu(B)^2+\eta}>1-\varepsilon .
\]
Let \(R=\operatorname{rank}(\Lambda)=r(d-1)+1\), where
\(r=\operatorname{rank}(\Gamma)\). Let \(\mathscr H\) be the haystack from
\Cref{Prop: Haystack_s_i}. Every element of \(\mathscr H\) has the form
\[
        \lambda=(s_1,\ldots,s_{d-1},1),
        \qquad s_1,\ldots,s_{d-1}\in\Gamma .
\]
By \Cref{Lem: bound_L_gamma}, there exists
\[
        M=M(\eta,R)=M(\delta,\varepsilon,\operatorname{rank}(\Gamma),d)
\]
such that, if
\[
        \sigma_B(\Rat_M(\Lambda))
        <
        \frac{\delta^2\varepsilon^2}{4}
        =
        \frac{\eta}{2},
\]
then there exists \(\lambda\in\cH_M\) such that
\[
        \sigma_B(L_\lambda^\perp\setminus\{1\})\leq\eta .
\]
By \Cref{Lem: Proj_Eig},
\[
        \sigma_B(\{1\})=\mu(B)^2 .
\]
Hence
\[
        \sigma_B(L_\lambda^\perp)
        \leq
        \mu(B)^2+\eta .
\]
Applying \Cref{Lem: sigma_B}, we obtain
\[
        \mu\left(\bigcup_{\tau\in\bZ} \tau\lambda.B\right)
        >
        \frac{\mu(B)^2}{\sigma_B(L_\lambda^\perp)}
        \geq
        \frac{\mu(B)^2}{\mu(B)^2+\eta}
        >
        1-\varepsilon .
\]
Since \(\lambda=(s_1,\ldots,s_{d-1},1)\) for some
\(s_1,\ldots,s_{d-1}\in\Gamma\), this proves the proposition.
\end{proof}


\subsection{Proof of the directional expansion theorem}

We now prove \Cref{Thm: Quant_direc_exp} from
\Cref{Prop: Dicotomy} and \Cref{Lem: measure_increment}.

\begin{proof}[Proof of \Cref{Thm: Quant_direc_exp}]
It is enough to consider \(0<\varepsilon<1\). Fix \(\delta>0\) and
\(\varepsilon>0\), and let \(\Lambda=\Gamma^{d-1}\times\bZ\), where
\(\Gamma<\bR\) is finitely generated. Let
\(\Lambda\curvearrowright (X,\mu)\) be an ergodic measure-preserving action,
and let \(B\subset X\) be measurable with \(\mu(B)\geq\delta\). Set
\[
        \theta=\frac{\delta^2\varepsilon^2}{4}.
\]
If there exist \(s_1,\ldots,s_{d-1}\in\Gamma\) such that
\[
        \mu\left(
        \bigcup_{\tau\in\bZ} \tau(s_1,\ldots,s_{d-1},1).B
        \right)
        >
        1-\varepsilon ,
\]
then the conclusion holds with \(k=1\) and \(\nu=\mu\). We may therefore
assume that
\[
        \mu\left(
        \bigcup_{\tau\in\bZ} \tau(s_1,\ldots,s_{d-1},1).B
        \right)
        \leq
        1-\varepsilon
\]
for every \(s_1,\ldots,s_{d-1}\in\Gamma\). By \Cref{Prop: Dicotomy}, there
exists
\[
        M=M(\delta,\varepsilon,\operatorname{rank}(\Gamma),d)\in\bN
\]
such that
\[
        \sigma_B(\Rat_M(\Lambda))\geq\theta .
\]
Applying \Cref{Lem: measure_increment} with this \(M\), we obtain an integer
\(k_1=k_1(M)\) and a \(k_1\Lambda\)-ergodic component \(\nu_1\) of \(\mu\)
such that
\[
        \nu_1(B)
        \geq
        \sqrt{\mu(B)^2+\sigma_B(\Rat_M(\Lambda))}
        \geq
        \sqrt{\mu(B)^2+\theta}.
\]
Since \(\mu(B)\leq 1\) and \(\theta\leq 1\), this gives
\[
        \nu_1(B)\geq \mu(B)+\frac{\theta}{3}.
\]
We now iterate the same argument. Suppose that at some stage we have a \(K\Lambda\)-ergodic component \(\nu\) of
\(\mu\), where \(K\in\bN\), and \(\nu(B)\geq\delta\). Initially
\(K=1\) and \(\nu=\mu\). If there exist \(s_1,\ldots,s_{d-1}\in\Gamma\) such that
\[
        \nu\left(
        \bigcup_{\tau\in K\bZ} \tau(s_1,\ldots,s_{d-1},1).B
        \right)
        >
        1-\varepsilon ,
\]
then the conclusion holds with \(k=K\). Otherwise, apply
\Cref{Prop: Dicotomy} to the \(K\Lambda\)-action. Since
\[
        K\Lambda=(K\Gamma)^{d-1}\times K\bZ
\]
and \(\operatorname{rank}(K\Gamma)=\operatorname{rank}(\Gamma)\), the same
choice of \(M\) applies. Thus the spectral measure of \(B\), now computed for
the \(K\Lambda\)-action on \((X,\nu)\), satisfies
\[
        \sigma_{\nu,B}(\Rat_M(K\Lambda))\geq \theta .
\]
Applying \Cref{Lem: measure_increment} to the \(K\Lambda\)-action, we obtain an
integer \(h=h(M)\) and an \(hK\Lambda\)-ergodic component \(\nu'\) of \(\nu\)
such that
\[
        \nu'(B)
        \geq
        \sqrt{\nu(B)^2+\theta}
        \geq
        \nu(B)+\frac{\theta}{3}.
\]
Thus, unless the desired conclusion has already been reached, each step
increases the measure of \(B\) by at least \(\theta/3\).

This process cannot continue indefinitely, since every component \(\nu\)
satisfies \(\nu(B)\leq 1\). Hence it stops after at most
\[
        L=\left\lceil\frac{3}{\theta}\right\rceil
\]
steps. At the stopping time, let \(k=K\), where \(K\) is the current subgroup index in
the iteration. Then we have a \(k\Lambda\)-ergodic component \(\nu\) of \(\mu\),
with \(\nu(B)\geq\mu(B)\), and elements \(s_1,\ldots,s_{d-1}\in\Gamma\) such
that
\[
        \nu\left(
        \bigcup_{\tau\in k\bZ} \tau(s_1,\ldots,s_{d-1},1).B
        \right)
        >
        1-\varepsilon .
\]
Moreover, each enlargement factor in the iteration depends only on \(M\), and
the number of steps is at most \(L\). Since \(M\) depends only on
\(\delta,\varepsilon,\operatorname{rank}(\Gamma)\), and \(d\), the final
integer \(k\) is bounded by a constant
\[
        k_0=k_0(\delta,\varepsilon,\operatorname{rank}(\Gamma),d).
\]
This proves the theorem.
\end{proof}


\section{The dynamical trace-spectrum theorem}
\label{sec:dynamical-trace-spectrum}

\subsection{Statement of the dynamical theorem}

Let \(d\geq 2\), let \(\Gamma<\bR\) be a finitely generated subgroup, and set
\[
        \Lambda=\Gamma^{d-1}\times\bZ\subseteq\bR^d .
\]
If \(v_1,\ldots,v_d\in\Lambda\), we write
\[
        A(v_1,\ldots,v_d)
        =
        [\,v_1\ \cdots\ v_d\,]\in\Mat_d(\bR)
\]
for the matrix with columns \(v_1,\ldots,v_d\). We also write
\[
        \mathbf{Tr}(v_1,\ldots,v_d)
        :=
        \mathbf{Tr}(A(v_1,\ldots,v_d)).
\]

Let \(\Lambda\curvearrowright (X,\mu)\) be an ergodic measure-preserving
action, and let \(B\subset X\) be measurable with \(\mu(B)>0\). We define the
\emph{dynamical trace spectrum} of \(B\) by
\[
        \cT(B)
        =
        \left\{
        \mathbf{Tr}(v_1,\ldots,v_d):
        v_1,\ldots,v_d\in\Lambda
        \text{ and }
        \mu(B\cap v_1.B\cap\cdots\cap v_d.B)>0
        \right\}.
\]

\begin{theorem}\label{Thm: dynamical_trace_spectrum}
Let \(\delta>0\), let \(d\geq 2\), and let
\(\Lambda=\Gamma^{d-1}\times\bZ\), where \(\Gamma<\bR\) is a finitely generated
subgroup. Then there exists
\[
        q=q(\delta,\operatorname{rank}(\Gamma),d)\in\bN
\]
such that the following holds. For every ergodic measure-preserving action
\(\Lambda\curvearrowright (X,\mu)\) and every measurable set \(B\subset X\)
with \(\mu(B)\geq\delta\),
\[
        q^2\Gamma\times q^3\Gamma\times\cdots\times q^d\Gamma
        \subseteq
        \cT(B).
\]
\end{theorem}

The proof combines the directional expansion theorem from
\Cref{Thm: Quant_direc_exp} with the algebraic trace model from
\Cref{sec:algebraic-trace-model}. We first find configurations whose edge
matrices have the form \(Q(s,m,n,t)\). The trace calculation from
\Cref{Lem: finite_index_trace_lattice} then gives the required product
subgroup.


\subsection{Finding the required configurations}

We first connect the model matrices from \Cref{sec:algebraic-trace-model} to
configurations in \(\Lambda\). Let
\[
        s=(s_1,\ldots,s_{d-1})\in\Gamma^{d-1},
        \qquad
        m=(m_1,\ldots,m_{d-1})\in\bZ^{d-1},
\]
let \(r=(r_1,\ldots,r_{d-1})\in\Gamma^{d-1}\), and let \(t\in\bZ\). Put
\[
        u=(s_1,\ldots,s_{d-1},1)\in\Lambda .
\]
If \(Q=Q(s,m,r,t)\), then the columns of \(Q\) are
\[
        Qe_j=m_jA_se_{j+1},\qquad 1\leq j\leq d-1,
\]
and
\[
        Qe_d=(r_1,\ldots,r_{d-1},0)+tu .
\]
Thus \(Q\) is the edge matrix of the ordered tuple
\[
        \bigl(a,\,
        a+Qe_1,\,
        \ldots,\,
        a+Qe_d\bigr)
\]
for any \(a\in\Lambda\).

The next lemma is the dynamical input needed for the trace-spectrum theorem.

\begin{lemma}\label{Lem: model_configurations}
Let \(\delta>0\). Then there exist
\[
        h_0=h_0(\delta,\operatorname{rank}(\Gamma),d)\in\bN
        \qquad\text{and}\qquad
        L=L(\delta,\operatorname{rank}(\Gamma),d)\in\bN
\]
with the following property: For every ergodic measure-preserving action
\(\Lambda\curvearrowright (X,\mu)\) and every measurable set \(B\subset X\)
with \(\mu(B)\geq\delta\), there exist an integer \(1\leq h\leq h_0\), elements
\(s_1,\ldots,s_{d-1}\in\Gamma\), and non-zero integers
\(m_1,\ldots,m_{d-1}\), with
\[
        m_j\in h\bZ
        \qquad\text{and}\qquad
        |m_j|\leq L,
        \qquad 1\leq j\leq d-1,
\]
such that, for every \(r=(r_1,\ldots,r_{d-1})\in h\Gamma^{d-1}\), there exists
\(t\in h\bZ\) for which, with
\[
        s=(s_1,\ldots,s_{d-1}),
        \qquad
        m=(m_1,\ldots,m_{d-1}),
        \qquad
        Q=Q(s,m,r,t),
\]
one has
\[
        \mu\bigl(
        B\cap Qe_1.B\cap\cdots\cap Qe_d.B
        \bigr)>0 .
\]
\end{lemma}

\begin{proof}
We first record the repeated form of \Cref{Lem: Poincare_Quant} that will be
used. Choose numbers
\[
        \alpha_0=\frac{\delta}{2},
        \qquad
        \alpha_{i+1}=\frac{\alpha_i^2}{2},
        \qquad 0\leq i\leq d-2,
\]
and set
\[
        \varepsilon_0=\alpha_{d-1}.
\]
Let \(\ell_0\) be the maximum of the constants supplied by
\Cref{Lem: Poincare_Quant} for \(\alpha_0,\ldots,\alpha_{d-2}\). Then the
following holds: If an abelian group \(G\) acts by measure-preserving
transformations on a probability space \((X,\nu)\), if \(C\subset X\) satisfies
\(\nu(C)\geq\delta\), and if \(w_1,\ldots,w_{d-1}\in G\), then there exist
non-zero integers \(\ell_1,\ldots,\ell_{d-1}\), with
\[
        |\ell_j|\leq \ell_0,
        \qquad 1\leq j\leq d-1,
\]
such that
\[
        \nu\bigl(
        C\cap \ell_1w_1.C\cap\cdots\cap \ell_{d-1}w_{d-1}.C
        \bigr)>\varepsilon_0 .
\]
Indeed, one applies \Cref{Lem: Poincare_Quant} successively to the
\(\bZ\)-actions generated by \(w_1,\ldots,w_{d-1}\). At each step the set to
which recurrence is applied has measure \(>\alpha_i\), and the resulting set is
contained in the intersection displayed above.

Apply \Cref{Thm: Quant_direc_exp} with parameters \(\delta\) and
\(\varepsilon_0\). Let
\[
        h_0=k_0(\delta,\varepsilon_0,\operatorname{rank}(\Gamma),d)
\]
be the corresponding bound. We obtain an integer \(1\leq h\leq h_0\), an
\(h\Lambda\)-ergodic component \(\nu\) of \(\mu\), and elements
\(s_1,\ldots,s_{d-1}\in\Gamma\) such that \(\nu(B)\geq\mu(B)\geq\delta\) and,
with
\[
        u=(s_1,\ldots,s_{d-1},1),
\]
one has
\[
        \nu\left(
        \bigcup_{\tau\in h\bZ} \tau u.B
        \right)>1-\varepsilon_0 .
\]
We now apply the repeated recurrence statement to the \(h\Lambda\)-action on
\((X,\nu)\), with \(C=B\), using the elements
\[
        hA_se_2,\ hA_se_3,\ \ldots,\ hA_se_d
        \in h\Lambda .
\]
Thus there exist non-zero integers \(\ell_1,\ldots,\ell_{d-1}\), with
\[
        |\ell_j|\leq\ell_0,
        \qquad 1\leq j\leq d-1,
\]
such that, after setting
\[
        m_j=h\ell_j,\qquad 1\leq j\leq d-1,
\]
the set
\[
        B'
        =
        B\cap m_1A_se_2.B
        \cap\cdots\cap
        m_{d-1}A_se_d.B
\]
satisfies
\[
        \nu(B')>\varepsilon_0 .
\]
Let now \(r=(r_1,\ldots,r_{d-1})\in h\Gamma^{d-1}\), and put
\[
        g_r=(r_1,\ldots,r_{d-1},0)\in h\Lambda .
\]
Since \(g_r\in h\Lambda\) and \(\nu\) is \(h\Lambda\)-invariant,
\[
        \nu\left(
        \bigcup_{\tau\in h\bZ} (g_r+\tau u).B
        \right)
        =
        \nu\left(
        g_r.\bigcup_{\tau\in h\bZ}\tau u.B
        \right)
        >
        1-\varepsilon_0 .
\]
Together with \(\nu(B')>\varepsilon_0\), this implies that there exists
\(t\in h\bZ\) such that
\[
        \nu\bigl(B'\cap (g_r+tu).B\bigr)>0 .
\]
Since \(\nu\) is an \(h\Lambda\)-ergodic component of \(\mu\), it is absolutely
continuous with respect to \(\mu\). Hence
\[
        \mu\bigl(B'\cap (g_r+tu).B\bigr)>0 .
\]
By the definitions of \(B'\), \(g_r\), and \(u\), this says
\[
        \mu\bigl(
        B\cap m_1A_se_2.B
        \cap\cdots\cap
        m_{d-1}A_se_d.B
        \cap ((r_1,\ldots,r_{d-1},0)+tu).B
        \bigr)>0 .
\]
The vectors appearing here are precisely the columns of \(Q(s,m,r,t)\). Thus
\[
        \mu\bigl(
        B\cap Qe_1.B\cap\cdots\cap Qe_d.B
        \bigr)>0 .
\]
Finally, since \(1\leq h\leq h_0\) and \(|\ell_j|\leq\ell_0\), we have
\[
        |m_j|\leq h_0\ell_0,
        \qquad 1\leq j\leq d-1.
\]
Thus the lemma holds with \(L=h_0\ell_0\).
\end{proof}


\subsection{Proof of the dynamical theorem}

We now prove \Cref{Thm: dynamical_trace_spectrum}. Let \(\delta>0\), let
\(d\geq 2\), and let \(\Lambda=\Gamma^{d-1}\times\bZ\), where
\(\Gamma<\bR\) is finitely generated. Let \(\Lambda\curvearrowright (X,\mu)\)
be an ergodic measure-preserving action, and let \(B\subset X\) be measurable
with \(\mu(B)\geq\delta\).

Apply \Cref{Lem: model_configurations}. We obtain an integer
\(1\leq q\leq q_0(\delta,\operatorname{rank}(\Gamma),d)\), elements
\(s_1,\ldots,s_{d-1}\in\Gamma\), and non-zero integers
\(m_1,\ldots,m_{d-1}\), with
\[
        m_j\in q\bZ
        \qquad\text{and}\qquad
        |m_j|\leq L(\delta,\operatorname{rank}(\Gamma),d),
        \qquad 1\leq j\leq d-1,
\]
such that, for every \(r=(r_1,\ldots,r_{d-1})\in q\Gamma^{d-1}\), there exists
\(t\in q\bZ\) such that, with \(s=(s_1,\ldots,s_{d-1})\),
\(m=(m_1,\ldots,m_{d-1})\), and \(Q=Q(s,m,r,t)\), one has
\[
        \mu\bigl(
        B\cap Qe_1.B\cap\cdots\cap Qe_d.B
        \bigr)>0 .
\]
It follows from the definition of \(\cT(B)\) that
\[
        (c_2(Q),\ldots,c_d(Q))\in\cT(B)
\]
for every such \(r\).

We now determine the set of trace tuples obtained as \(r\) ranges over
\(q\Gamma^{d-1}\). By \Cref{Lem: ck-formula-Q}, for \(2\leq k\leq d-1\),
\[
        c_k(Q(s,m,r,t))
        =
        (-1)^{k-1}
        \bigl(r_{d-k+1}+s_{d-k}m_{d-k}\bigr)
        \prod_{j=d-k+1}^{d-1}m_j,
\]
and
\[
        c_d(Q(s,m,r,t))
        =
        (-1)^{d-1}r_1\prod_{j=1}^{d-1}m_j .
\]
Since \(m_j\in q\bZ\) for every \(j\), we have
\(s_{d-k}m_{d-k}\in q\Gamma\). Hence, as \(r\) ranges over
\(q\Gamma^{d-1}\), the trace tuples obtained contain the product subgroup
\[
        q m_{d-1}\Gamma
        \times
        q m_{d-2}m_{d-1}\Gamma
        \times\cdots\times
        q\left(\prod_{j=1}^{d-1}m_j\right)\Gamma .
\]
Thus this product subgroup is contained in \(\cT(B)\).

Choose an integer \(\kappa_0\) divisible by
\[
        q,\quad |m_1|,\quad\ldots,\quad |m_{d-1}|.
\]
For example, one may take
\[
        \kappa_0=q\prod_{j=1}^{d-1}|m_j|.
\]
Then, for every \(2\leq k\leq d\), the integer \(\kappa_0^k\) is divisible by
\[
        q\prod_{j=d-k+1}^{d-1}|m_j|,
\]
where for \(k=d\) the product is \(\prod_{j=1}^{d-1}|m_j|\). Therefore
\[
        \kappa_0^k\Gamma
        \subseteq
        q\left(\prod_{j=d-k+1}^{d-1}m_j\right)\Gamma,
        \qquad 2\leq k\leq d.
\]
Consequently,
\[
        \kappa_0^2\Gamma\times \kappa_0^3\Gamma
        \times\cdots\times \kappa_0^d\Gamma
        \subseteq
        \cT(B).
\]

Finally, since \(1\leq q\leq q_0(\delta,\operatorname{rank}(\Gamma),d)\) and
\(|m_j|\leq L(\delta,\operatorname{rank}(\Gamma),d)\), we may choose a single
integer
\[
        \kappa=\kappa(\delta,\operatorname{rank}(\Gamma),d)
\]
which is divisible by every possible value of \(\kappa_0\) arising from the
above construction. Replacing \(\kappa_0\) by this \(\kappa\), we get
\[
        \kappa^2\Gamma\times \kappa^3\Gamma\times\cdots\times \kappa^d\Gamma
        \subseteq
        \cT(B).
\]
Renaming \(\kappa\) as \(q\) proves \Cref{Thm: dynamical_trace_spectrum}.


\section{Proofs of the main theorems}

We now deduce the two main results from the dynamical trace-spectrum theorem.

\subsection{The discrete density theorem}

\begin{proof}[Proof of \Cref{thm:discrete-main}]
Let \(E\subseteq\bZ^d\) have positive upper Banach density, and set
\[
        \delta=d^*(E)>0.
\]
We identify \(\bZ^d\) with \(\Gamma^{d-1}\times\bZ\), where \(\Gamma=\bZ\).
By \Cref{Prop: FCP_Lambda}, there exist an ergodic measure-preserving action
\(\bZ^d\curvearrowright (X,\mu)\) and a measurable set \(B\subset X\), with
\(\mu(B)=\delta\), such that for every \(v_1,\ldots,v_d\in\bZ^d\),
\[
        \mu\bigl(
        B\cap v_1.B\cap\cdots\cap v_d.B
        \bigr)
        \leq
        d^*_{\bZ^d}\bigl(
        E\cap(E-v_1)\cap\cdots\cap(E-v_d)
        \bigr).
\]
Apply \Cref{Thm: dynamical_trace_spectrum} with \(\Gamma=\bZ\). There exists
\(q=q(\delta,d)\in\bN\) such that
\[
        q^2\bZ\times q^3\bZ\times\cdots\times q^d\bZ
        \subseteq
        \cT(B).
\]
Let
\[
        (a_2,\ldots,a_d)
        \in
        q^2\bZ\times q^3\bZ\times\cdots\times q^d\bZ .
\]
By the definition of \(\cT(B)\), there exist \(v_1,\ldots,v_d\in\bZ^d\) such
that
\[
        \mathbf{Tr}([\,v_1\ \cdots\ v_d\,])=(a_2,\ldots,a_d)
\]
and
\[
        \mu\bigl(
        B\cap v_1.B\cap\cdots\cap v_d.B
        \bigr)>0 .
\]
The correspondence principle therefore gives
\[
        d^*_{\bZ^d}\bigl(
        E\cap(E-v_1)\cap\cdots\cap(E-v_d)
        \bigr)>0 .
\]
In particular, the intersection is non-empty. Choose
\(x\in E\cap(E-v_1)\cap\cdots\cap(E-v_d)\), and set
\[
        \mathbf v=(x,x+v_1,\ldots,x+v_d)\in E^{d+1}.
\]
Then
\[
        A_{\mathbf v}=[\,v_1\ \cdots\ v_d\,],
\]
and hence
\[
        \mathbf{Tr}(A_{\mathbf v})=(a_2,\ldots,a_d).
\]
Since \((a_2,\ldots,a_d)\) was arbitrary, we have
\[
        q^2\bZ\times q^3\bZ\times\cdots\times q^d\bZ
        \subseteq
        \TraceSpec(E).
\]
This proves \Cref{thm:discrete-main}.
\end{proof}

\subsection{The Euclidean colouring theorem}

We shall use the following product theorem of Graham
\cite[Theorem 3]{Graham_1980}. Let \(G\) be an abelian group. A family
\(\mathscr F\) of finite subsets of \(G\) is called a \emph{Ramsey family} if,
for every finite colouring \(G=G_1\sqcup\cdots\sqcup G_r\), there exist
\(F\in\mathscr F\), \(g\in G\), and \(1\leq i\leq r\) such that
\[
        F+g\subseteq G_i .
\]

\begin{theorem}[Graham's product theorem]\label{Thm: Graham_product}
Let \(G\) be an abelian group, and let
\((\mathscr F_\alpha)_{\alpha\in I}\) be a collection of Ramsey families of
finite subsets of \(G\). Then, for every finite colouring
\(G=G_1\sqcup\cdots\sqcup G_r\), there exists \(1\leq i_0\leq r\) such that,
for every \(\alpha\in I\), there exist \(F_\alpha\in\mathscr F_\alpha\) and
\(g_\alpha\in G\) with
\[
        F_\alpha+g_\alpha\subseteq G_{i_0}.
\]
\end{theorem}

If \(\mathbf v=(v_0,\ldots,v_m)\) is an ordered tuple in \(G\), we write
\[
        \supp(\mathbf v)=\{v_0,\ldots,v_m\}
\]
for its underlying finite set. Thus a monochromatic translate of
\(\supp(\mathbf v)\) gives a monochromatic translate of the ordered tuple
\(\mathbf v\).

\begin{proof}[Proof of \Cref{thm:euclidean-main}]
Let
\[
        \bR^d=C_1\sqcup\cdots\sqcup C_r
\]
be a finite colouring. We first show that each prescribed trace tuple is
realized monochromatically in some colour.

Fix
\[
        a=(a_2,\ldots,a_d)\in\bR^{d-1},
\]
and set \(\delta=1/r\). Choose an integer \(\kappa\) divisible by
\[
        q(\delta,\rho,d),
        \qquad 1\leq \rho\leq d,
\]
where \(q(\delta,\rho,d)\) denotes the integer supplied by
\Cref{Thm: dynamical_trace_spectrum} for finitely generated subgroups of
\(\bR\) of rank \(\rho\). Define
\[
        \Gamma
        =
        \left\langle
        1,\frac{a_2}{\kappa^2},\frac{a_3}{\kappa^3},\ldots,
        \frac{a_d}{\kappa^d}
        \right\rangle
        <\bR
\]
and set \(\Lambda=\Gamma^{d-1}\times\bZ\subseteq\bR^d\). Then
\(\operatorname{rank}(\Gamma)\leq d\). Restrict the colouring to \(\Lambda\).
Since
\[
        \Lambda=(C_1\cap\Lambda)\sqcup\cdots\sqcup(C_r\cap\Lambda),
\]
there exists \(1\leq i\leq r\) such that
\[
        d^*_\Lambda(C_i\cap\Lambda)\geq \frac1r=\delta .
\]
Set \(E=C_i\cap\Lambda\). By \Cref{Prop: FCP_Lambda}, there exist an ergodic
measure-preserving action \(\Lambda\curvearrowright(X,\mu)\) and a measurable
set \(B\subset X\), with \(\mu(B)=d^*_\Lambda(E)\geq\delta\), such that the
correspondence inequality holds for all finite intersections.

Apply \Cref{Thm: dynamical_trace_spectrum} to this action. Let \(q\) be the
integer supplied by that theorem. Since \(q\) divides \(\kappa\), and since
\[
        a_k\in \kappa^k\Gamma\subseteq q^k\Gamma,
        \qquad 2\leq k\leq d,
\]
we get
\[
        a=(a_2,\ldots,a_d)\in\cT(B).
\]
Thus there exist \(v_1,\ldots,v_d\in\Lambda\) such that
\[
        \mathbf{Tr}([\,v_1\ \cdots\ v_d\,])=a
\]
and
\[
        \mu\bigl(
        B\cap v_1.B\cap\cdots\cap v_d.B
        \bigr)>0 .
\]
The correspondence principle gives
\[
        d^*_\Lambda\bigl(
        E\cap(E-v_1)\cap\cdots\cap(E-v_d)
        \bigr)>0 .
\]
Hence this intersection is non-empty. Choose
\(x\in E\cap(E-v_1)\cap\cdots\cap(E-v_d)\), and set
\[
        \mathbf v=(x,x+v_1,\ldots,x+v_d)\in C_i^{d+1}.
\]
Then
\[
        A_{\mathbf v}=[\,v_1\ \cdots\ v_d\,],
\]
and therefore
\[
        \mathbf{Tr}(A_{\mathbf v})=a.
\]
This proves that, for every \(a\in\bR^{d-1}\), some colour class contains an
ordered \((d+1)\)-tuple with trace tuple \(a\).

We now use Graham's product theorem to make the colour independent of \(a\).
For each \(a\in\bR^{d-1}\), let \(\mathscr F_a\) be the family of all finite
sets of the form \(\supp(\mathbf v)\), where
\[
        \mathbf v=(v_0,\ldots,v_d)\in(\bR^d)^{d+1}
        \qquad\text{and}\qquad
        \mathbf{Tr}(A_{\mathbf v})=a .
\]
The previous paragraph shows that \(\mathscr F_a\) is a Ramsey family in the
additive group \(\bR^d\). Applying \Cref{Thm: Graham_product} to the collection
\((\mathscr F_a)_{a\in\bR^{d-1}}\), there exists a colour \(C_{i_0}\) such
that, for every \(a\in\bR^{d-1}\), there exist
\(\mathbf v_a=(v_{a,0},\ldots,v_{a,d})\) with
\(\mathbf{Tr}(A_{\mathbf v_a})=a\) and \(g_a\in\bR^d\) such that
\[
        g_a+\supp(\mathbf v_a)\subseteq C_{i_0}.
\]

Since \(g_a+\supp(\mathbf v_a)\subseteq C_{i_0}\), the ordered tuple
\[
        g_a+\mathbf v_a
        :=
        (g_a+v_{a,0},\ldots,g_a+v_{a,d})
\]
belongs to \(C_{i_0}^{d+1}\). Since translating all entries of an ordered tuple
does not change its edge matrix, we have
\[
        \mathbf{Tr}(A_{g_a+\mathbf v_a})
        =
        \mathbf{Tr}(A_{\mathbf v_a})
        =
        a.
\]
Thus, for every \(a\in\bR^{d-1}\), the colour class \(C_{i_0}\) contains an
ordered tuple \(\mathbf w\in C_{i_0}^{d+1}\) such that
\[
        \mathbf{Tr}(A_{\mathbf w})=a.
\]
Hence
\[
        \TraceSpec(C_{i_0})=\bR^{d-1}.
\]
This proves \Cref{thm:euclidean-main}.
\end{proof}

\section{The linear trace obstruction}

\subsection{The Euclidean obstruction}

\begin{proof}[Proof of \Cref{obs:euclidean-trace}]
We first prove the two-dimensional statement. Let
\[
        I=\left[0,\frac{3}{7}\right)^2
        \subseteq \bR^2.
\]
Partition the half-open square \([0,3)^2\) into \(49\) translates of \(I\), and
give each translate a different colour. Then extend this colouring periodically
by translations in \(3\bZ^2\). Thus each colour class has the form
\[
        C=(x_0,y_0)+I+3\bZ^2
\]
for some \((x_0,y_0)\in[0,3)^2\).

Fix one colour class \(C\), and let
\[
        \mathbf v=(v_0,v_1,v_2)\in C^3.
\]
Write \(v_j=(x_j,y_j)\). Since all three points lie in the same colour class,
there exist \(p_j,q_j\in\bZ\) and \(\xi_j,\eta_j\in[0,3/7)\) such that
\[
        x_j=x_0+\xi_j+3p_j,
        \qquad
        y_j=y_0+\eta_j+3q_j,
        \qquad j=0,1,2.
\]
The edge matrix is
\[
        A_{\mathbf v}
        =
        \begin{pmatrix}
        x_1-x_0 & x_2-x_0\\
        y_1-y_0 & y_2-y_0
        \end{pmatrix},
\]
and therefore
\[
        \trace(A_{\mathbf v})
        =
        (x_1-x_0)+(y_2-y_0).
\]
Using the decompositions above, we get
\[
        \trace(A_{\mathbf v})
        \in
        \left(-\frac{6}{7},\frac{6}{7}\right)+3\bZ .
\]
In particular,
\[
        1\notin
        \{\trace(A_{\mathbf v}):
        \mathbf v=(v_0,v_1,v_2)\in C^3\}.
\]
This proves the two-dimensional obstruction. Since the same conclusion holds
even without assuming affine independence, it also holds for the subset of
affinely independent triples.

For \(d\geq 2\), the same construction gives a finite colouring of \(\bR^d\).
Partition \([0,3)^d\) into half-open cubes of side length \(3/N\), where
\(N>3d\), give these cubes distinct colours, and extend periodically by
\(3\bZ^d\). If \(\mathbf v=(v_0,\ldots,v_d)\) is monochromatic, then each
diagonal entry of \(A_{\mathbf v}\) belongs to
\[
        \left(-\frac{3}{N},\frac{3}{N}\right)+3\bZ.
\]
Hence
\[
        \trace(A_{\mathbf v})
        \in
        \left(-\frac{3d}{N},\frac{3d}{N}\right)+3\bZ.
\]
Since \(N>3d\), this set does not contain \(1\). Thus no colour class realizes
all values of \(c_1(A_{\mathbf v})=\trace(A_{\mathbf v})\). This proves the
Euclidean obstruction in every dimension.
\end{proof}

\subsection{The discrete obstruction}

\begin{proof}[Proof of \Cref{obs:discrete-trace}]
Let \(\alpha\in\bR/\bZ\) be irrational, and write
\[
        \|x\|_{\bR/\bZ}
\]
for the distance from \(x\) to \(0\) in \(\bR/\bZ\). Define
\[
        E
        =
        \left\{
        n=(n_1,\ldots,n_d)\in\bZ^d:
        \|n_j\alpha\|_{\bR/\bZ}<\frac{1}{6d}
        \text{ for every } 1\leq j\leq d
        \right\}.
\]
By equidistribution, \(E\) has positive upper Banach density.

Let \(\mathbf v=(v_0,\ldots,v_d)\in E^{d+1}\), and write
\(v_i=(v_{i,1},\ldots,v_{i,d})\). Then
\[
        \trace(A_{\mathbf v})
        =
        \sum_{j=1}^d (v_{j,j}-v_{0,j}).
\]
Since both \(v_j\) and \(v_0\) belong to \(E\), we have
\[
        \|(v_{j,j}-v_{0,j})\alpha\|_{\bR/\bZ}
        <
        \frac{1}{3d},
        \qquad 1\leq j\leq d.
\]
Therefore
\[
        \|\trace(A_{\mathbf v})\alpha\|_{\bR/\bZ}
        <
        \frac13 .
\]
It follows that
\[
        \{\trace(A_{\mathbf v}):
        \mathbf v\in E^{d+1}\}
        \subseteq
        \left\{
        n\in\bZ:
        \|n\alpha\|_{\bR/\bZ}<\frac13
        \right\}.
\]
Now let \(q\in\bN\). Since \(\alpha\) is irrational, the subgroup
\[
        q\bZ\alpha
        =
        \{qn\alpha:n\in\bZ\}
\]
is dense in \(\bR/\bZ\). Hence there exists \(n\in q\bZ\) such that
\[
        \|n\alpha\|_{\bR/\bZ}\geq \frac13 .
\]
Thus \(q\bZ\) is not contained in the trace set above. Since \(q\) was
arbitrary, the set of linear traces does not contain any finite-index subgroup
of \(\bZ\). The same conclusion holds after restricting to affinely independent
tuples, and this proves \Cref{obs:discrete-trace}.
\end{proof}

\crefname{appendix}{appendix}{appendices}
\Crefname{appendix}{Appendix}{Appendices}

\appendix
\renewcommand{\thesection}{\Alph{section}}

\section{Trace calculations}
\label[appendix]{sec: trace_calc}

In this appendix we prove \Cref{Lem: ck-formula-Q}. Recall that
\[
        Q=Q(s,m,n,t)=A_sM_{m,n,t}.
\]
Its columns have the following form. For \(1\leq j\leq d-2\),
\[
        Qe_j=m_je_{j+1},
\]
while
\[
        Qe_{d-1}
        =
        m_{d-1}(s_1,\ldots,s_{d-1},1)^T
\]
and
\[
        Qe_d
        =
        (n_1+s_1t,\ldots,n_{d-1}+s_{d-1}t,t)^T .
\]
We shall compute the principal minors of \(Q\). For
\(I\subseteq\{1,\ldots,d\}\), let \(Q_I\) denote the principal submatrix with
rows and columns indexed by \(I\). Since, for \(1\leq j\leq d-2\), the
\(j\)-th column of \(Q\) has its only possibly non-zero entry in row \(j+1\),
we have
\[
        \det(Q_I)=0
        \qquad\text{if } j\in I,\ j+1\notin I
        \text{ for some } 1\leq j\leq d-2.
\]
It follows that, for \(2\leq k\leq d-1\), the only principal minors of order
\(k\) which can be non-zero are those corresponding to
\[
        I_1=\{d-k,d-k+1,\ldots,d-1\}
\]
and
\[
        I_2=\{d-k+1,d-k+2,\ldots,d\}.
\]
Hence
\[
        c_k(Q)=\det(Q_{I_1})+\det(Q_{I_2}),
        \qquad 2\leq k\leq d-1.
\]

We first compute \(\det(Q_{I_1})\). Put \(p=d-k\). Then
\(I_1=\{p,p+1,\ldots,d-1\}\). The relevant submatrix is
\[
        Q_{I_1}
        =
        \begin{pmatrix}
        0 & 0 & \cdots & 0 & s_p m_{d-1}\\
        m_p & 0 & \cdots & 0 & s_{p+1}m_{d-1}\\
        0 & m_{p+1} & \cdots & 0 & s_{p+2}m_{d-1}\\
        \vdots & \vdots & \ddots & \vdots & \vdots\\
        0 & 0 & \cdots & m_{d-2} & s_{d-1}m_{d-1}
        \end{pmatrix}.
\]
The only non-zero term in its determinant expansion is obtained by taking the
top-right entry, together with the subdiagonal entries. Thus
\[
        \det(Q_{I_1})
        =
        (-1)^{k-1}s_p m_{d-1}
        \prod_{j=p}^{d-2}m_j .
\]
Since \(p=d-k\), this is
\[
        \det(Q_{I_1})
        =
        (-1)^{k-1}
        s_{d-k}m_{d-k}
        \prod_{j=d-k+1}^{d-1}m_j .
\]
We now compute \(\det(Q_{I_2})\). Put \(p=d-k+1\), so that
\(I_2=\{p,p+1,\ldots,d\}\). Inside this \(k\times k\) submatrix, the first
\(k-2\) columns have the subdiagonal entries
\[
        m_p,m_{p+1},\ldots,m_{d-2}.
\]
The last two columns are
\[
        m_{d-1}u
        \qquad\text{and}\qquad
        n'+tu,
\]
where
\[
        u=(s_p,\ldots,s_{d-1},1)^T,
        \qquad
        n'=(n_p,\ldots,n_{d-1},0)^T .
\]
By multilinearity of the determinant, the contribution from the term \(tu\) in
the last column vanishes, since it is linearly dependent with the preceding
column \(m_{d-1}u\). Therefore
\[
        \det(Q_{I_2})
        =
        \det(C_1,\ldots,C_{k-2},m_{d-1}u,n'),
\]
where \(C_1,\ldots,C_{k-2}\) denote the first \(k-2\) columns.

In any non-zero term of this determinant, we see that the columns
\(C_1,\ldots,C_{k-2}\) must use the rows \(2,\ldots,k-1\). The remaining rows
are \(1\) and \(k\). Since the last entry of \(n'\) is \(0\), the last column
must use row \(1\), contributing \(n_p\), while the column \(m_{d-1}u\) uses
row \(k\), contributing \(m_{d-1}\). Hence
\[
        \det(Q_{I_2})
        =
        (-1)^{k-1}
        n_p
        \prod_{j=p}^{d-1}m_j .
\]
Since \(p=d-k+1\), this gives
\[
        \det(Q_{I_2})
        =
        (-1)^{k-1}
        n_{d-k+1}
        \prod_{j=d-k+1}^{d-1}m_j .
\]
Combining the two computations, for \(2\leq k\leq d-1\) we obtain
\[
        c_k(Q)
        =
        (-1)^{k-1}
        \bigl(n_{d-k+1}+s_{d-k}m_{d-k}\bigr)
        \prod_{j=d-k+1}^{d-1}m_j .
\]
It remains to compute \(c_d(Q)=\det(Q)\). This is the same calculation as the
one for \(I_2\), with \(p=1\) and \(k=d\). Thus
\[
        \det(Q)
        =
        (-1)^{d-1}
        n_1
        \prod_{j=1}^{d-1}m_j .
\]
This proves \Cref{Lem: ck-formula-Q}.

\end{document}